\documentclass{amsart}
\usepackage{amsfonts,amssymb,amscd,amsmath,enumerate,verbatim,calc}
\usepackage{amscd,amssymb,amsopn,amsmath,amsthm,graphics,amsfonts,enumerate,verbatim,calc}
\usepackage[dvips]{graphicx}
\usepackage[colorlinks=true,linkcolor=red,citecolor=blue]{hyperref}
\input xy
\xyoption{all}
\calclayout

\newcommand{\rt}{\rightarrow}
\newcommand{\lrt}{\longrightarrow}

\newcommand{\st}{\stackrel}

\newcommand{\al}{\alpha}

\newcommand{\La}{\Lambda}
\newcommand{\Ga}{\Gamma}
\newcommand{\lan}{\langle}
\newcommand{\ran}{\rangle}

\newcommand{\C}{\mathbb{C} }
\newcommand{\D}{\mathbb{D} }
\newcommand{\K}{\mathbb{K} }

\newcommand{\Z}{\mathbb{Z} }
\newcommand{\R}{\mathbb{R}}

\newcommand{\CA}{\mathcal{A} }
\newcommand{\CC}{\mathcal{C} }

\newcommand{\CL}{\mathcal{L} }
\newcommand{\CM}{\mathcal{M} }

\newcommand{\CQ}{\mathcal{Q} }

\newcommand{\CS}{\mathcal{S} }
\newcommand{\CT}{\mathcal{T} }
\newcommand{\CX}{\mathcal{X} }
\newcommand{\CY}{\mathcal{Y} }

\newcommand{\CB}{\mathcal{B}}

\newcommand{\X}{\mathbf{X}}
\newcommand{\Y}{\mathbf{Y}}

\newcommand{\A}{\mathbf{A}}

\newcommand{\Mod}{{\rm{Mod\mbox{-}}}}

\newcommand{\mmod}{{\rm{{mod\mbox{-}}}}}
\newcommand{\Fmod}{\rm mod}
\newcommand{\FMod}{\rm Mod}

\newcommand{\Inj}{{\rm{Inj}\mbox{-}}}
\newcommand{\inj}{{\rm{inj}\mbox{-}}}
\newcommand{\Prj}{{\rm{Prj}\mbox{-}}}
\newcommand{\prj}{{\rm{prj}\mbox{-}}}
\newcommand{\GPrj}{{\GP}\mbox{-}}
\newcommand{\Gprj}{{\Gp\mbox{-}}}
\newcommand{\GInj}{{\GI \mbox{-}}}
\newcommand{\Ginj}{{\Gi \mbox{-}}}

\newcommand{\Sum}{{\rm Sum}\mbox{-}}

\newcommand{\op}{{\rm{op}}}

\newcommand{\sg}{{\rm{sg}}}
\newcommand{\add}{{\rm{add}\mbox{-}}}

\newcommand{\Add}{{\rm{Add}\mbox{-}}}

\newcommand{\ac}{{\rm{ac}}}

\newcommand{\id}{{\rm{id}}}

\newcommand{\GP}{{\mathcal{GP}}}
\newcommand{\Gp}{{\mathcal{G}p}}
\newcommand{\GI}{{\mathcal{GI}}}
\newcommand{\Gi}{{\mathcal{G}i}}

\newcommand{\bb}{\rm b}

\newcommand{\Hom}{{\rm{Hom}}}
\newcommand{\Ext}{{\rm{Ext}}}
\newcommand{\End}{{\rm{End}}}

\newtheorem{theorem}{Theorem}[section]
\newtheorem{corollary}[theorem]{Corollary}
\newtheorem{lemma}[theorem]{Lemma}

\newtheorem{proposition}[theorem]{Proposition}

\theoremstyle{definition}

\newtheorem{example}[theorem]{Example}

\newtheorem{remark}[theorem]{Remark}
\newtheorem{setup}[theorem]{Setup}

\newtheorem{s}[theorem]{}

\theoremstyle{plain}

\theoremstyle{definition}

\numberwithin{equation}{section}

\begin{document}

\title[Relative derived categories]{On relative derived categories}

\author[Asadollahi, Bahiraei, Hafezi, Vahed]{J. Asadollahi, P. Bahiraei, R. Hafezi and R. Vahed}

\address{Department of Mathematics, University of Isfahan, P.O.Box: 81746-73441, Isfahan, Iran and School of Mathematics, Institute for Research in Fundamental Science (IPM), P.O.Box: 19395-5746, Tehran, Iran }
\email{asadollahi@ipm.ir and asadollahi@sci.ui.ac.ir}

\address{Department of Mathematics, University of Isfahan, P.O.Box: 81746-73441, Isfahan, Iran}
\email{pbahiraei.math@sci.ui.ac.ir}

\address{School of Mathematics, Institute for Research in Fundamental Science (IPM), P.O.Box: 19395-5746, Tehran, Iran }
\email{hafezi@ipm.ir}
\email{vahed@ipm.ir}



\thanks{This research was in part supported by a grant from IPM (No: 93130216). }

\begin{abstract}
The paper is devoted to study some of the questions arises naturally in connection to the notion of relative derived categories. In particular, we study invariants of recollements involving relative derived categories, generalise two results of Happel by proving the existence of AR-triangles in Gorenstein derived categories, provide situations for which relative derived categories with respect to Gorenstein projective and Gorenstein injective modules are equivalent and finally study relations between the Gorenstein derived category of a quiver and its image under a reflection functor. Some interesting applications are provided.
\end{abstract}

\maketitle

\tableofcontents

\section{Introduction}
Derived categories were invented by Grothendieck-Verdier \cite{Ve77} in the early sixties, simply to make quasi-isomorphisms become isomorphisms. Nowadays, derived categories are widely used in algebraic geometry, commutative algebra and also representation theory of algebras.

On the other hand, there has been an approach in homological algebra, going back to Eilenberg and Moore in their 1965 AMS Memoir \cite{EM} and Mac Lane \cite{Ma}, that replaces projective modules by relative projective modules. With this approach modules other than projectives are allowed in resolutions, and hence a relative version of homological algebra was grown up. The most important case, that goes back to Auslander and Bridger's 1964 AMS Memoir \cite{AB}, is to treat a class of relative projective modules known as Gorenstein projective modules.

The importance and applications of the notion of derived categories in many different subjects provide a natural motivation to study the relative version of derived categories. Buan \cite{Bu} studied relative derived categories by localizing relative quasi-isomorphisms. He applied these categories to generalize Happel's result on derived equivalences induced by tilting modules to the relative setting. Moreover, Gao and Zhang \cite{GZ} used Gorenstein homological algebra to study Gorenstein derived categories.  This notion then was studied more in \cite{AHV}.

This paper should be considered as a continuation of \cite{GZ} and \cite{AHV}, as further attempts to achieve more properties of relative derived categories, specially with respect to the class of Gorenstein projectives. The outcome is that this relative setting has nice properties one could expect and deserve more attentions. In fact, beside other applications, one hope is that such theories could help in understanding the Gorenstein Symmetry Conjecture, when one apply the theory to the module category of an artin algebra.

Let $\CA$ be an abelian category and $\CX$ be a subcategory of $\CA$. A complex ${\bf A}$ in $\CA$ is called $\CX$-acyclic if the induced complex $\Hom_{\CA}(X, \A)$ is acyclic for every $X \in \CX$. The Gorenstein derived category of $\CA$, denoted $\D_{\GP}(\CA)$, is by definition the Verdier quotient of the homotopy category of $\CA$, $\K(\CA)$, modulo the thick subcategory $\K_{\GP \mbox{-} {\rm ac}}(\CA)$ of $\GP(\CA)$-acyclic complexes, where $\GP$ denotes the class of Gorenstein projective objects. Some of the basic properties of $\D_{\GP}(\CA)$ are investigated in \cite{GZ}. In a similar way, one can define the $\CX$-relative derived category of $\CA$, denoted by $\D_{\CX}(\CA)$.

There are several open questions transferred naturally from (absolute) derived settings to relative settings. See \cite[Appendix 3]{Ch3} for a (not complete) list of open problems. In this paper we discuss some of such questions. Let us be more precise.

Section 2 is devoted to the preliminaries. In Section 3, we show that when $\La$ is an artin algebra, $\D_{\CX}(\Mod \La)$, for certain subcategories $\CX$ of $\Mod \La$, has a nice interpretation as the (absolute) derived category of a functor category, see Theorem \ref{D(ModS)}.

Recollements are important in algebraic geometry and representation theory, see for example \cite{BBD, Kon}. They were introduced by Beilinson, Bernstein and Deligne \cite{BBD} with the idea that one triangulated category can be considered as being `glued together' from two others.
Study of invariants of recollements also have been the subject of several researches. For example, Wiedemann \cite[Lemma 2.1]{W} proved that if we have a $\D^{\bb}(\mmod )$ level recollement of finite dimensional algebras over a field, writing $A$ in terms of $B$ and $C$, i.e.
\[ \xymatrix@C-0.5pc@R-0.5pc{ \D^{\bb}(\mmod B) \ar[rr] && \D^{\bb}(\mmod A) \ar[rr] \ar@/^0.75pc/[ll] \ar@/_0.75pc/[ll] && \D^{\bb}(\mmod C)\ar@/^0.75pc/[ll] \ar@/_0.75pc/[ll] }\]
then ${\rm gl. dim}(A)< \infty$ if and only if ${\rm gl. dim}(B)< \infty $ and ${\rm gl. dim}(C)< \infty$. Similar result for the finitistic dimension is proved by Happel \cite[3.3]{H}. In these cases, we say that the finiteness of global dimension and also finiteness of finitistic dimension are invariants of recollements. There is an example showing that being of finite representation type is not an invariant of recollements, see Section 4 below. As an application of our results, see Theorem \ref{MainThm}, we show that if we have a $\D^{\bb}_{\Gp}(\mmod )$ level recollement of virtually Gorenstein algebras, then being Gorenstein as well as being of finite Cohen-Macaulay type, are invariants of recollements.

The notion of Auslander-Reiten triangles, AR-triangles for short, is introduced by Happel \cite{Hap}. He established their existence in the bounded derived category of finitely generated $\La$-modules, $\D^{\bb}(\mmod \La)$, for an artin algebra $\La$. A natural attempt is to prove the existence theorem of AR-triangles for the bounded Gorenstein derived category of $\La$, $\D^{\bb}_{\GP}(\mmod \La)$. This question was first considered by Gao \cite{G12}, where she proved the existence of AR-triangles for any finite dimensional Gorenstein algebras of finite Cohen-Macaulay type. In Section 5, we extend this result to Gorenstein algebras. Furthermore, in Theorem \ref{AR-triangels}, we show that when $\La$ is a virtually Gorenstein algebra, then $\D^{\bb}_{\Gp}(\mmod \La)$ has Auslander-Reiten triangles if and only if $\La$ is Gorenstein. This result should be compared with \cite[Corollay 1.5]{Hap2}.

In Section 6, we study situations for which relative derived categories with respect to Gorenstein projective and Gorenstein injective modules are equivalent. We show that such a balance property exists, i.e. $\D_{\Gp}^{\bb}(\mmod \La) \simeq \D_{\Gi}^{\bb}(\mmod \La)$, provided $\La$ is a virtually Gorenstein  algebra, see Theorem \ref{ThmSec6}.

The notion of reflection functors were introduced by Bernstein, Gel'fand, and Ponomarev \cite{BGP} to relate representations of two quivers. Then this notion has been extended by several authors, see \cite{BB} and \cite{HR}. In the last section of the paper, Corollary  \ref{AHV2}, we prove that if $\La$ is a Gorenstein algebra, then $\D^{\bb}_{\Gp}(\CQ)$ is equivalent to $\D^{\bb}_{\Gp}(\sigma_i \CQ)$ via reflection functors. This is a relative version of a result of Happel \cite[I. 5.7]{Hap}. Two direct corollaries are the equivalences of triangulated categories $\D_{\sg}(\CQ) \cong \D_{\sg}(\sigma_i\CQ)$ and $\underline{\Gp}\mbox{-}\CQ \simeq \underline{\Gp}\mbox{-}\sigma_i\CQ,$ where $\D_{\sg}(\CQ)$ denotes the singularity category of $\La\CQ$ and $\underline{\Gp}$ denotes the stable category of Gorenstein projective modules with respect to projectives.

Let us end this introduction by mentioning that Theorems \ref{AR-triangels} and \ref{MainThm} in this paper, are proved independently in \cite{Gao2}.

\section{Preliminaries}
Let $A$ be a ring, associative with identity and  $\Mod A$, resp. $\mmod A$, denote the category of all, resp. all finitely presented, $A$-modules.

For an $A$-module $M$, $\add M$ is the additive subcategory of $\Mod A$ consisting of all direct summands of finite direct sums of copies of $M$. For a class $\CM$  of $A$-modules, $\Sum \CM$, resp. ${\rm sum}\mbox{-} \CM$, denotes the subcategory of $\Mod A$ consisting of all, resp. all finite, direct sums of objects in $\CM$.

Let $\CA$ be an additive category. $\C(\CA)$ denotes the category of complexes over $\CA$. We write complexes cohomologically, so every object of  $\C(\CA)$ is of the form
$$ \cdots \lrt A^{n-1} \lrt A^n \lrt A^{n+1} \lrt \cdots.$$
$\K(\CA)$ denotes  the homotopy category of $\CA$. Moreover,  $\K^-(\CA)$, resp. $\K^{\bb}(\CA)$, denotes the full subcategory of $\K(\CA)$ formed by all bounded above, resp. bounded, complexes. As well, we let $\K^{-, \bb}(\CA)$ denote the full subcategory of $\K^{-}(\CA)$ consisting of all complexes ${\bf X}$ in which there is an integer $n=n_{\X}$ such that ${\rm H}^i(\X)=0$, for all $i \leq n$.

The derived category of $\CA$ will be denoted by $\D(\CA)$. We write $\D^{-}(\CA)$, resp $\D^{\bb}(\CA)$, for the full subcategory of $\D(\CA)$ consisting of all homologically bounded above, resp. homologically bounded, complexes.

Let $\CA$ be an abelian category and $\CX \subseteq \CA$ be a full additive subcategory of $\CA$ which is closed under direct summands. For every object $C$ of $\CA$, a morphism $f: X \lrt C$ is called a right $\CX$-approximation of $C$, if $X \in \CX$ and for every $X' \in \CX$ the induced map $\Hom_{\CA}(X' ,X) \lrt \Hom_{\CA}(X' , C)$ is surjective. The subcategory $\CX$ of $\CA$ is called contravariantly finite if every object in $\CA$ admits a right $\CX$-approximation.

\s Let $\CX$ be an additive category. A complex $\X \in \C(\CX)$ is called $\CX$-totally acyclic if for every $X \in \CX$, the induced complexes $\Hom(X, \X)$ and $\Hom(\X, X)$ are acyclic.
Let $\CA$ be an abelian category having enough projective, resp. injective, objects and $\CX= \Prj \CA$, resp. $\CX= \Inj \CA$, be the class of projectives, resp. injectives. In this case, an $\CX$-totally acyclic complex is called totally acyclic complex of projectives, resp. totally acyclic complex of injectives. An object $G$ in $\CA$ is called Gorenstein projective, resp. Gorenstein injective, if $G$ is a  syzygy of a totally acyclic complex of projectives, resp. totally acyclic complex of injectives. We denote the class of all Gorenstein projective, resp. Gorenstein injective, objects in $\CA$ by $\GPrj \CA$, resp. $\GInj \CA$. In case $\CA= \Mod A$, we abbreviate the notations to $\GPrj A$ and $\GInj A$. We set $\Gprj A= \GPrj A \cap \mmod A$ and $\Ginj A= \GInj A \cap \mmod A$.

\s {\sc Relative derived category.}  Let $\CA$ be an abelian category and $\CX$ be a subcategory of $\CA$. A complex ${\bf A}$ in $\CA$ is called $\CX$-acyclic (sometimes $\CX$-proper exact) if  the induced complex $\Hom_{\CA}(X, \A)$ is acyclic for every $X \in \CX$.
A chain map $f:{\bf A} \lrt {\bf B}$ in $\C(\CA)$ is called an $\CX$-quasi-isomorphism if  the induced chain map $\CA(X,f)$ is a quasi-isomorphism, for any $X \in \CX$.

The relative derived category of $\CA$, denoted by $\D^*_{\CX}(\CA)$, is  the Verdier quotient of the homotopy category $\K^*(\CA)$ with respect to  the thick triangulated subcategory $\K^*_{\CX\mbox{-}\ac}(\CA)$ of $\CX$-acyclic complexes, where $* \in \{{\rm blank},-,\bb\}$; see \cite[\S 2]{GZ} and \cite[\S 3]{Ch1}.

If $\CA$ has enough projective objects and $\CX = \GPrj \CA$, then $\D_{\CX}(\CA)$ is known as the Gorenstein derived category of $\CA$, denoted by $\D_{\GP}(\CA)$. The Gorenstein derived category was first introduced and studied by Gao and Zhang \cite{GZ}. In case $\CX=\Gprj \La$, where $\La$ is a ring, we write  $\K_{\Gp \mbox{-}\ac}(\mmod\La)$ for $\K_{\CX\mbox{-}\ac}(\CA)$ and $\D_{\Gp}(\mmod\La)$ for $\D_{\CX}(\CA)$.

\s{\sc Functor categories.}\label{FCat}
Let $\CC$ be a skeletally small additive category and $\CB$ be an additive category. The additive covariant, resp. contravariant,  functors from $\CC$ to $\CB$ together with the natural transformations and their compositions of functors form a category, known as the functor category $(\CC, \CB)$, resp.\ $(\CC^{\op}, \CB)$. We are mainly interested in the category $(\CC^{\op}, \CA b)$ of abelian group valued functors. This category inherits a notion of exactness from the category of abelian groups. A sequence $F' \st{u} \lrt F \st{v} \lrt F''$ of functors is called exact if for every $X \in \CC$,  the induced sequence $F'(X) \st{u_X}\lrt F(X) \st{v_X} \lrt F''(X)$ is an exact sequence of abelian groups. This, in particular, implies that $(\CC^{\op}, \CA b)$ is an abelian category.

By using Yoneda lemma one can see that, for every $C \in \CC$, the functor $\Hom_{\CC}(-, C)$ is a projective object in the functor category $(\CC^{\op}, \CA b)$. Moreover, every object  $F$ in $(\CC^{\op},\CA b)$ is covered by an epimorphism $ \coprod_{i}\Hom_{\CC}(-,C_i) \lrt F \lrt 0$, where $i$ runs through isomorphism classes of objects of $\CC$.  So  the category $(\CC^{\op},\CA b)$ has enough projective objects.

Following \cite{A}, we denote the category $( \CC^{\op}, \CA b)$ by $\FMod (\CC)$, which is called the category of modules on $\CC$.
A $\CC$-module $F$ is called finitely presented if there is an exact sequence
$$ \Hom_{\CC}(-, X) \lrt \Hom_{\CC}(-,Y) \lrt F \lrt 0$$
of $\CC$-modules, with $X$ and $Y$ in $\CC$. The category of all finitely presented $\CC$-modules is denoted by $\Fmod(\CC)$.

A morphism $C_2 \st{f}\rt C_1$ is called pseudokernel for a morphism $C_1 \st{g}\rt C_0$ if the sequence
$$ \Hom_{\CC}(-,C_2) \st{\Hom_{\CC}(-,f)} \lrt \Hom_{\CC}(-,C_1) \st{\Hom_{\CC}(-,g)} \lrt \Hom_{\CC}(-,C_0)$$
of functors is exact. It is known that $\CC$ has pseudokernels if and only if $\Fmod (\CC)$ is abelian.

\s {\sc Recollements.}
Let $\CT'$, $\CT$ and $\CT''$ be triangulated categories.
 A recollement of  $\CT$  with respect to  $\CT'$ and $\CT''$ is defined by six triangulated functors as follows
\[\xymatrix{\CT'\ar[rr]^{i_*=i_!}  && \CT \ar[rr]^{j^*=j^!} \ar@/^1pc/[ll]_{i^!} \ar@/_1pc/[ll]_{i^*} && \CT'' \ar@/^1pc/[ll]_{j_*} \ar@/_1pc/[ll]_{j_!} }\]
satisfying the following conditions:
\begin{itemize}
\item[$(i)$] $(i^*,i_*)$, $(i_!,i^!)$, $(j_!, j^!)$ and $(j^*,j_*)$ are adjoint pairs.
\item[$(ii)$] $i^!j_*=0$, and hence $j^!i_!=0$ and $i^*j_!=0$.
\item[$(iii)$] $i_*$, $j_*$ and $j_!$ are full embeddings.
\item[$(iv)$] for any object $T \in \CT$, there exist the following triangles
\[i_!i^!(T) \rt  T \rt j_*j^*(T) \rightsquigarrow \ \ \ \text{and} \ \ \ j_!j^!(T) \rt T \rt i_*i^*(T) \rightsquigarrow\]
in $\CT$.
\end{itemize}

\s {\sc Thick subcategories.}
Let $\CT$ be a triangulated category. A triangulated subcategory $\CL$ of $\CT$ is called thick, if it is closed under direct summands. Let $L$ be a set of objects in $\CT$. We denote by ${\rm thick}(L)$ the smallest thick subcategory of $\CT$ containing $L$. It is a known fact that, for a ring $R$, the subcategory ${\rm thick}(R)$ of $\D^{\bb}(\mmod R)$ is just $\K^{\bb}(\prj R)$.

\section{A description of relative derived categories}
Let $\La$ be an artin algebra. In this section, we show that $\D_{\CX}(\Mod \La)$, for certain subcategories $\CX$ of $\Mod \La$, has a description as the (absolute) derived category of a functor category.

It is proved in Corollary 3.6 of \cite{AHV} that if $M \in \mmod \La$, then there is an equivalence $\D^{\bb}_{\Add M}(\Mod \La) \simeq \D^{\bb}(\Mod \End_{\La}(M))$ of triangulated categories. In the following proposition we extend this equivalence to the set $\CS$ contained in $\mmod \La$. This result also generalizes Theorem 3.1 of \cite{G1}.

\begin{proposition}\label{D(ModS)}
Let $\La$ be an artin algebra and $\CS$ be a set contained in $\mmod \La$. Set $\CX := \Add \CS$. Then there is an equivalence $$\D_{\CX}(\Mod \La) \simeq \D(\FMod (\CS))$$ of triangulated categories.
\end{proposition}

\begin{proof}
Set $\CA= \Mod \La$. Let $\lan \CS \ran$ be the smallest full triangulated subcategory of $\K(\CA)$ that contains $\CS$ and is closed under coproducts. Let $\X$ be a complex in ${}^\perp \K_{\CX \mbox{-} \ac}(\CA)$. Then $\Hom_{\K(\CA)}(f, \X)$ is bijection for all $\CX$-quasi-isomorphisms $f$. Thus, the quotient functor $\K(\CA) \lrt \D_{\CX}(\CA)$ induces a bijection $\Hom_{\K(\CA)}(\X, \Y) \cong \Hom_{\D_{\CX}(\CA)}(\X, \Y)$, where $\X, \Y \in {}^\perp \K_{\CX\mbox{-}\ac}(\CA)$; see \cite[Proposition 5-3]{Ve77} and \cite[Lemma 4.8.1]{Kra3}. Since $\lan \CS \ran \subseteq {}^{\perp}\K_{\CX\mbox{-}\ac}(\CA)$, we may deduce that the composition
\[\varphi: \lan \CS \ran \hookrightarrow \K(\CA) \lrt \D_{\CX}(\CA)\]
of functors is fully faithful. Since $\lan \CS \ran$ is compactly generated, by \cite[Theorem 4.1]{N}, the inclusion $\lan \CS \ran \hookrightarrow \K(\CA)$ has a right adjoint. Hence for every complex $\X \in \K(\CA)$, there exists an exact triangle
$\Y \lrt \X \lrt {\bf Z} \rightsquigarrow$, in which $\Y \in \lan \CS \ran$ and ${\bf Z} \in {\lan \CS \ran}^{\perp}$. But ${\lan \CS \ran}^{\perp} \subseteq \K_{\CX\mbox{-}\ac}(\CA)$. Hence in $\D_{\CX}(\CA)$, $\Y \cong\X$. This, in turn, implies that $\varphi$ is dense and so is an equivalence.

On the other hand, let $(-, \CS)$ denote the class of all functors $\Hom_{\La}(-,S)$, where $S \in \CS$. Since $\FMod(\CS)$ is an abelian category with a projective generator and exact coproducts, then  $\D(\FMod(\CS)) \simeq \lan (-,\CS) \ran $, see \cite{K07}. So we just need to prove that $\lan \CS \ran \simeq \lan (-,\CS) \ran $. To this end, consider the functor $\phi : \Sum \CS \lrt \Sum (-,\CS)$ given by $\phi ( \oplus_{i \in I}S_i)= \oplus_{i \in I}\Hom_{\La}(-,S_i)$. This functor is, in fact, an equivalence. Indeed, it follows from the following isomorphisms that $\phi$ is fully faithful.
\[\begin{array}{llll}
\Hom_{\La}(\oplus_{i\in I} S_i, \oplus_{j \in J}S_j) & \cong \prod_{i \in I} \Hom_{\La}(S_i, \oplus_{j \in J}S_j)
\\ & \cong \prod_{i \in I} \oplus_{j \in J} \Hom_{\La}(S_i, S_j)
\\ & \cong \prod_{i \in I} \oplus_{j \in J} \Hom (\Hom_{\La}(-, S_i), \Hom_{\La}(-, S_j))
\\ & \cong \Hom(\oplus_{i \in I} \Hom_{\La}(-,S_i), \oplus_{j \in J}\Hom_{\La}(-,S_j)).
\end{array}\]
Note that the second isomorphism follows from the fact that finitely generated $\La$-modules are compact objects in $\mmod \La$ and the last isomorphism follows from the same fact applying to  $\Hom_{\La}(-,S_i)$ in the category $\FMod (\CS)$. Moreover, $\phi$ is clearly dense and so is an equivalence of categories.

Naturally, $\phi$ can be extended to an equivalence of homotopy categories $$\bar{\phi}: \K(\Sum \CS) \lrt \K(\Sum (-,\CS)).$$
Observe that starting from any object of $\K(\Add \CS)$ we can get an object of $\K(\Sum \CS)$ by taking the direct sum with complexes of the form
$$\cdots \lrt 0 \lrt 0 \lrt S \st{\id}\lrt S \lrt 0 \lrt 0 \lrt \cdots,$$
 that are zero objects in the homotopy category. So $\K(\Sum \CS) \simeq \K(\Add \CS)$.  \\In a similar way, one can deduce that $\K(\Sum (-,\CS)) \simeq \K(\Prj \FMod(\CS))$. Consequently, we have an equivalence $\K(\CX) \simeq \K(\Prj \Fmod(\CS))$.

Now, our desired equivalence follows from  the following diagram
\[\xymatrix{ \K(\CX) \ar[r]^{\sim \ \ \ \ \ \ } \ar@{<-_{)}}[d] & \K(\Prj \FMod(\CS)) \ar@{<-_{)}}[d] \\ \lan \CS \ran  \ar[r]^{\sim \ \ } & \lan (-,\CS) \ran.}\]
Hence the proof is complete.
\end{proof}

Recall that an artin algebra $\La$ is called virtually Gorenstein, if $(\GPrj \La)^\perp = {}^\perp (\GInj \La)$, see \cite{BR}. Moreover, an artin algebra $\La$ is called of finite Cohen-Macaulay type (CM-finite, for short) provided there are only finitely many indecomposable finitely generated Gorenstein projective $\La$-modules, up to isomorphism, see \cite{BR} and \cite{Be2}.

\begin{corollary}
Let $\La$ be a virtually Gorenstein algebra of finite CM-type. Then $$\D_{\GP}(\Mod \La) \simeq \D (\FMod(\Gprj \La)).$$
\end{corollary}

\begin{proof}
By \cite[Theorem 4.10]{Be3}, $\GPrj \La= \Add \Gprj \La$. Therefore, Proposition \ref{D(ModS)} applies to conclude the result.
\end{proof}

\section{Invariants of recollements of Gorenstein derived categories}
Wiedemann \cite[Lemma 2.1]{W} proved that if we have a $\D^{\bb}(\mmod )$ level recollement of finite dimensional algebras over a field, writing $A$ in terms of $B$ and $C$,
then ${\rm gl. dim}(A)$ is finite if and only if ${\rm gl. dim}(B)$ and ${\rm gl. dim}(C)$ are finite. Similar result for the finitistic dimension is proved by Happel \cite[3.3]{H}. That is finiteness of global dimension and also finiteness of finitistic dimension are invariants of recollements.

In this section, we show that being CM-finite and also being Gorenstein are invariants of recollements of Gorenstein derived categories over certain artin algebras.  Note that, as the following examples show, being of finite representation type or being Gorenstein are not invariants of recollements of (absolute) derived categories.

\begin{example}
Let $k$ be an algebraically closed filed and $\La= k[x]/(x^n)$ with $n>5$.
It is proved in \cite[Corollary 7.7]{P} that for any ring $\La$ there is the following recollement
\[\xymatrix@C=0.5cm{ \D^{\bb}(\mmod \La) \ar[rrr]^{i_*}  &&& \D^{\bb}(\mmod T_2(\La)) \ar[rrr]^{j^*} \ar@/^1.5pc/[lll]_{i^!} \ar@/_1.5pc/[lll]_{i^*} &&& \D^{\bb}(\mmod \La),\ar@/^1.5pc/[lll]_{j_*} \ar@/_1.5pc/[lll]_{j_!} }\]
where $T_2(\La)= \left(\begin{array}{ll} \La & 0 \\ \La & \La \end{array} \right)$ is the $T_2$-extension of $\La$. By \cite[Example 4.17]{Be3}, we know that $T_2(\La)$ if of infinite CM-type, while it is known that $\La$ is of finite representation type.
\end{example}

\begin{example}
Let $A$ be a $k$-algebra, where $k$ is a filed, given by the quiver
\[\xymatrix{   e_1 \ar[r]^{\al} & e_2 \ar@(dr,ur)[]_{\beta}}\]
bound by the zero relations $\al \beta =0$ and $\beta^2=0$. Consider an idempotent $e_1$ of $A$. In view of \cite[Example 4.5]{PSS}, there is the following  recollement
\vspace{0.3cm}
\[\xymatrix@C=0.5cm{ \D^{\bb}(\mmod A/Ae_1A) \ar[rrr] &&& \D^{\bb}(\mmod A) \ar[rrr] \ar@/^1.5pc/[lll] \ar@/_1.5pc/[lll] &&& \D^{\bb}(\mmod k).\ar@/^1.5pc/[lll] \ar@/_1.5pc/[lll] }\]
\vspace{0.3cm}
\\It can easily checked that $A/Ae_1A$ is a $k$-algebra given by the quiver
\[\xymatrix{    e_2 \ar@(dr,ur)[]_{\beta}}\]
bound by the zero relation $\beta^2=0$. Therefore, $A/Ae_1A$, as well as $k$, is Gorenstein, while $A$ is not Gorenstein, see \cite[Example 4.3(2)]{Ch}.
\end{example}

\begin{setup}\label{setup}
Throughout this and also next section, we assume that the artin algebra $\La$ satisfies the following property
\begin{center}
$(\divideontimes)$ $\Gprj \La$ is a contravariantly finite subcategory of $\mmod \La$.
\end{center}
\end{setup}

\begin{remark}\label{contra}
Note that if $\La$ is a virtually Gorenstein algebra or is an artin algebra of finite CM-type, then it satisfies property $(\divideontimes)$, i.e. $\Gprj \La$ is a contravariantly finite subcategory of $\mmod \La$. This fact for virtually Gorenstein algebras is proved by Beligiannis \cite[Proposition 4.7]{Be3} and for CM-finite algebras is easy to see because in this case $\Gprj\La = \add G$, for some $G \in \Gprj \La$. So our results at least covers these two classes of algebras.

It is an open problem that weather CM-finite algebras are virtually Gorenstein, see \cite[Page 52]{Ch3}.
\end{remark}

Let $\K^{-, \Gp \bb}(\Gprj \La)$ denote the full subcategory of $\K^-(\Gprj \La)$ consisting of all complexes $\X$ such that there is an integer $n =n_{\X}$ such that  ${\rm H}^i(\Hom_{\La}(G,\X))=0$, for all $i \leq n$ and every $G \in \Gprj \La$.
\vspace{0.2cm}

\begin{lemma}\label{Rem}
Let $\La$ be an artin algebra and  $\X$ be a complex in $\K^{-,\Gp \bb}(\Gprj \La)$. Then  $\X$ lies in $\K^{\bb}(\Gprj \La)$ if and only if for every $\Y \in \K^{-,\Gp \bb}(\Gprj \La)$, $\Hom_{\K^{-,\Gp\bb}(\Gprj \La)}(\X, \Y[n])=0$, for large $n$.
\end{lemma}

\begin{proof}
If $\X$ is isomorphic to a bounded complex of finitely generated Gorenstein projective $\La$-module, then $\Hom(\X, \Y[i])=0$, for large $i$. Conversely, assume that $\X$ is not isomorphic to a complex in $\K^{\bb}(\Gprj \La)$. So, for small $i$, the kernel of the $i$-th differential of $\X$ is not finitely generated Gorenstein projective. Take $\Y$ to be the direct sum of Gorenstein projective resolutions of all simple $\La$-modules. Then $\Hom(\X, \Y[i])\neq 0$ for large $i$ and this is a contradiction.
\end{proof}

\begin{remark}
Let $\La$ be an artin algebra and $M \in \mmod \La$. A proper Gorenstein projective resolution of $M$ is a sequence
$$ \cdots \lrt G_n \lrt G_{n-1} \lrt \cdots \lrt G_0 \lrt M \lrt 0,$$
which is a $\Gp$-acyclic complex.

Let ${\bf G} \lrt M$ be a proper Gorenstein projective resolution of $M$. For every module $N \in \mmod \La$ and every integer $n$, the relative Ext group, denoted by $\Ext_{\mathcal{G}}(-,-)$, is defined by
$$ \Ext_{\mathcal{G}}^n(M, N) := {\rm H}^n(\Hom_{\La}({\bf G}, N)).$$

In view of Theorem 3.12 of \cite{GZ}, if $\La$ is a virtually Gorenstein algebra, then
$$\Hom_{\D^{\bb}_{\Gp}(\mmod \La)}(M , N[n]) \cong \Ext^n_{\mathcal G}(M,N),$$
for every $M,N \in \mmod \La$ and every $n \in \Z$.
\end{remark}

Now we are ready to prove our first result in this section.

\begin{theorem}\label{MainThm}
Let $A,B$ and $C$ be artin algebras satisfying property $(\divideontimes)$. Assume that $\D^{\bb}_{\Gp}(\mmod A)$ admits the following recollement
\[ \xymatrix@C=0.5cm{ \D^{\bb}_{\Gp}(\mmod B) \ar[rrr]^{i_*}  &&& \D^{\bb}_{\Gp}(\mmod A) \ar[rrr]^{j^*} \ar@/^1.5pc/[lll]_{i^!} \ar@/_1.5pc/[lll]_{i^*} &&& \D^{\bb}_{\Gp}(\mmod C).\ar@/^1.5pc/[lll]_{j_*} \ar@/_1.5pc/[lll]_{j_!} }\]
Then
\begin{itemize}
\item[$(i)$] $A$ is of finite CM-type if and only if $B$ and $C$ are so.
\item [$(ii)$] $A$ is Gorenstein if and only if $B$ and $C$ are so.
\end{itemize}
\end{theorem}

\begin{proof}
$(i)$ Observe that by our assumption, $\Gprj A$, $\Gprj B$ and $\Gprj C$ are contravariantly finite subcategories of $\mmod A$, $\mmod B$ and $\mmod C$, respectively. So, by \cite[Theorem 3.3]{AHV} the given recollement can be stated in the following form
\[ \xymatrix@C=0.5cm{ \K^{-, \Gp\bb}(\Gprj B) \ar[rrr]^{i_*}  &&& \K^{-,\Gp \bb}(\Gprj A) \ar[rrr]^{j^*} \ar@/^1.5pc/[lll]_{i^!} \ar@/_1.5pc/[lll]_{i^*} &&& \K^{-, \Gp\bb}(\Gprj C).\ar@/^1.5pc/[lll]_{j_*} \ar@/_1.5pc/[lll]_{j_!} }\]
\vspace{0.1cm}

Let $A$ be of finite CM-type. So there is an $A$-module $G$ such that $\Gprj A=\add G$. Observe that the functor $i_*$, resp. $j_!$, sends any finitely generated Gorenstein projective $B$-module $G'$, resp. $C$-module $G''$, to a bounded complex over $\Gprj A$. Indeed, for any $\Y \in \K^{-,\Gp \bb}(\Gprj A)$, $\Hom_{\K^{-,\Gp \bb}(\Gprj A)}(i_*(G'), \Y[n])$, resp. $\Hom_{\K^{-,\Gp \bb}(\Gprj A)}(j_!(G''),\Y[n])$, is isomorphic to \\$\Hom_{\K^{-,\Gp \bb}(\Gprj B)}(G', i^! \Y[n])$, resp. $\Hom_{\K^{-,\Gp \bb}(\Gprj C)}(G'', j_* \Y[n])$, and hence vanishes for large $n$. So Lemma \ref{Rem} implies that $i_*(G')$ and $j_!(G'')$ belong to $\K^{\bb}(\Gprj A)$.

Since $\add G$ generates $\K^{\bb}(\Gprj A)$, the natural isomorphism $i^*i_* \cong {\rm id}_{\K^{-,\Gp \bb}(\Gprj B)}$ implies that $\Gprj B = \add (\oplus_{n \in \Z}i^*(G)^n)$, where $i^*(G)^n$ denotes the $n$-th term of the complex $i^*(G)$. Note that $i^*$ has a right adjoint and so it can be easily checked, by using the above remark, that $i^*(G)$ is in $\K^{\bb}(\Gprj B)$. So the direct sum $\oplus_{n \in \Z} i_*(G)^n$ has only finitely many non-zero terms and hence $B$ is of finite CM-type.
Similarly, from the unit of the adjoint pair $(j_!, j^*)$ we obtain the
isomorphism ${\rm id}_{\K^{-,\Gp \bb}(\Gprj C)}\cong j^*j_!$ and so $\Gprj C = \add (\oplus_{n \in \Z} j^*(G)^n)$. Again since $j^*$ possesses a right adjoint, $j^*(G)$ is bounded. Hence, $C$ is of finite CM-type.

For the converse, let $G'$, resp. $G''$, be a $B$-module, resp. $C$-module, in which $ \Gprj B=\add G'$, resp. $\Gprj C=\add G''$. Then, $\K^{\bb}(\Gprj B)= {\rm thick}( G')$ and $\K^{\bb}(\Gprj C)={\rm thick }(G'')$.  Let $G$ be a finitely generated Gorenstein projective $A$-module. There exists the standard triangle
$$ j_!j^* (G) \rt G \rt i_*i^*(G) \rightsquigarrow,$$
induced from the above recollement. Since $j^*$ has a right adjoint, the above argument can be applied to show that $j^*(G)$ lies in $\K^{\bb}(\Gprj C)$. Hence the above triangle implies that $G$ belongs to ${\rm thick}( i_*(G') \oplus j_!(G''))$. So, $i_*(G') \oplus j_!(G'')$ generates the bounded homotopy category $\K^{\bb}(\Gprj A)$.

Therefore, every complex $\X \in \K^{\bb}(\Gprj A)$ is homotopy equivalent to a bounded complex $\Y$ with all terms in $\add (\bar{G'}\oplus \bar{G''})$, where $\bar{G'}= \oplus_{n \in \Z}i_*(G')^n$ and $\bar{G''}=\oplus_{n \in \Z}j_!(G'')^n$. Observe that as we saw above $i_*(G')$ and $j_!(G'')$ lie in $\K^{\bb}(\Gprj A)$. Thus, the direct sum $\oplus_{n \in \Z}i_*(G')^n$ as well as the direct sum $\oplus_{n \in \Z}j_!(G'')^n$ have only finitely many non-zero terms. Now, if we pick a finitely generated Gorenstein projective $A$-module $G$ and consider it as a stalk complex with $G$ in degree zero, then $G$ is homotopy equivalent  to a bounded complex $\Y$ with all terms in $\add (\bar{G'}\oplus \bar{G''})$. This, in turn, implies that $G$ should belong to $\add (\bar{G'}\oplus \bar{G''})$ and hence $A$ is of finite CM-type.

$(ii)$
Let $A$ be Gorenstein. Then $\D^{\bb}_{\Gp}(\mmod A) \simeq \K^{\bb}(\Gprj A)$.
Since $\D^{\bb}_{\Gp}(\mmod B)$ is fully embedded in $\D^{\bb}_{\Gp}(\mmod A)$, by Lemma \ref{Rem}, every complex in $\D^{\bb}_{\Gp}(\mmod B)$ is isomorphic to a bounded complex in $\K^{\bb}(\Gprj B)$. Therefore, by \cite[Theorem 1.18]{CFH}, $B$ is Gorenstein. The same argument can be applied to see that $C$ is also Gorenstein.

For the converse, let $S$ and $T$ be simple $A$-modules.  The above recollement induces the following triangles in $\D^{\bb}(\Fmod(\Gprj A))$,
\[ i_*i^! T \rt T \rt j_*j^*T \rightsquigarrow;\]
\[ j_!j^* S \rt S \rt i_*i^* S \rightsquigarrow.\]
Now, apply the homological functor $\Hom_{\D^{\bb}_{\Gp}(\mmod A)}(S, -)$ on the first  and
the  cohomological functors
 $$\Hom_{\D^{\bb}_{\Gp} (\mmod A)} ( -, i_*i^!T)\  \text{and}\  \Hom_{\D^{\bb}_{\Gp}(\mmod A)}\ (-, j_*j^*T)$$ on the second triangle to obtain the following long exact sequences
 \[\xymatrix@R=0.5cm@C=0.5cm{ \cdots\ar[r] & \Hom(S, i_*i^!T[i]) \ar[r] & \Hom(S, T[i]) \ar[r] & \Hom(S, j_*j^* T[i]) \ar[r] & \cdots;\\
 \cdots \ar[r] & \Hom( i_*i^*S, i_*i^!T[i]) \ar[r] & \Hom( S, i_*i^!T[i]) \ar[r] & \Hom( j_!j^*S, i_*i^!T[i]) \ar[r] & \cdots; \\
 \cdots \ar[r] & \Hom(i_*i^*S, j_*j^*T[i]) \ar[r] & \Hom(S, j_*j^*T[i]) \ar[r] & \Hom(j_!j^*S, j_*j^*T[i]) \ar[r] & \cdots.}\]
Since $B$ is Gorenstein, $i^*S$ and $i^!T$ lie in $\K^{\bb}(\Gprj B)$. Thus
$\Hom(i_*i^*S, i_*i^!T[i])=0$, for $i \gg0$. Similarly, since $C$ is Gorenstein, $\Hom(j_!j^*S, j_*j^*T[i])=0$ for $i \gg 0$. Moreover, by the adjointness properties  $\Hom( j_!j^*S, i_*i^!T[i])=0 = \Hom(i_*i^*S, j_*j^*T[i])$ for every $i \in \Z$.

Consequently, the above three exact sequences imply that $$\Ext_{\mathcal G}^i(S, T) \cong \Hom_{\D^{\bb}_{\Gp}(\mmod(\Gprj A))}(S, T[i])=0 $$ for $i\gg0$. Therefore, by \cite[Themorem 12.1.4]{EJ}, every simple, and hence every finitely generated, $A$-module has finite Gorenstein projective dimension. Now, Theorem 1.18 of \cite{CFH} yields that $A$ is Gorenstein.
\end{proof}

The following corollary generalises a result due to Wiedemann \cite{W}, in terms of functor categories. We preface it with a Lemma.

\begin{lemma}\label{Gor.ness}
Let $\La$ be an artin algebra satisfying the property $(\divideontimes)$. Then $\Fmod(\Gprj \La)$ has finite global dimension if and only if $\La$ is Gorenstein.
\end{lemma}

\begin{proof}
Let $\La$ be an $n$-Gorenstein algebra. Consider a functor $F$ in $\Fmod(\Gprj \La)$, i.e. there is an exact sequence
$$ (-, X_1) \st{(-,d_1)} \lrt (-, X_0) \lrt F \lrt 0$$
with $X_1$ and $X_0$ in $\Gprj \La$. Let $X$ be the kernel of $d_1$. Then there is an exact sequence $$0 \lrt G \lrt X_{n-1} \lrt \cdots \lrt X_2 \lrt X \lrt 0$$
such that for every $i \in \{ 2, \cdots, n-1\}$, $X_i$ is a projective $\La$-module and $G$ is a Gorenstein projective $\La$-module. This exact sequence gives a projective resolution
$$ 0 \lrt (-, G) \lrt (-, X_{n-1}) \lrt \cdots \lrt (-, X_2) \lrt (-, X_1) \lrt (-, X_0) \lrt F \lrt 0 $$
of $F$. This means that ${\rm gl.dim} (\Fmod (\Gprj \La)) \leq n$. For the `only if' part, let $M$ be a finitely generated $\La$-module and ${\rm gl. dim } (\Fmod(\Gprj \La))= n$. Since $\Gprj \La$ is cotravariantly finite, there is a $\Gp$-proper exact resolution
$$ \cdots \lrt G_{n-1} \lrt \cdots \lrt G_1 \lrt G_0 \lrt M \lrt 0$$
of $M$, where  $G_i \in \Gprj \La$ for all $i$. One should apply Yoneda Lemma to see that the $n$-th syzygy of the above resolution is Gorenstein projective. So every finitely generated $\La$-module has finite Gorenstein projective dimension. Now \cite[Theorem 1.8]{CFH} implies that $\La$ is Gorenstein.
\end{proof}

\begin{corollary}\label{Wied}
Let $A$, $B$ and $C$ be artin algebras satisfying property $(\divideontimes)$. Assume that there exists the following recollement
\[ \xymatrix@C=0.5cm{ \D^{\bb}(\Fmod(\Gprj B)) \ar[rrr]^{i_*}  &&& \D^{\bb}(\Fmod(\Gprj A)) \ar[rrr]^{j^*} \ar@/^1.5pc/[lll]_{i^!} \ar@/_1.5pc/[lll]_{i^*} &&& \D^{\bb}(\Fmod(\Gprj C)).\ar@/^1.5pc/[lll]_{j_*} \ar@/_1.5pc/[lll]_{j_!} }\]
\vspace{0.1cm}
Then $\Fmod(\Gprj A)$ has finite global dimension if and only if $\Fmod(\Gprj B)$ and $\Fmod(\Gprj C)$ have finite global dimensions.
\end{corollary}

\begin{proof}
It follows directly from Lemma \ref{Gor.ness} and Theorem \ref{MainThm}$(ii)$.
\end{proof}

\section{Auslander-Reiten triangles in $\D^{\bb}_{\Gp}(\mmod \La)$}
A known result due to Happel \cite{Hap2} states that for a finite dimensional algebra $\La$, $\D^{\bb}(\mmod \La)$ has Auslander-Reiten triangles if and only if $\La$ has finite global dimension. In this section we present a Gorenstein version of this result \cite[Corollary 1.5]{Hap2}. More precisely, we prove that if $\La$ satisfies property $(\divideontimes)$, i.e. if $\Gprj \La$ is a contravariantly finite subcategory of $\mmod \La$, then $\D_{\Gp}^{\bb}(\mmod \La)$ has Auslander-Reiten triangles if and only if $\La$ is Gorenstein. This also generalises the main result of \cite{G12} that states $\D^{\bb}_{\GP}(\mmod \La)$ has Auslander-Reiten triangles, provided $\La$ is a finite dimensional Gorenstein algebra of finite CM-type.

Recall that a map $\al:X \rt Y$ is called left almost split, if $\al$ is not a section and if every map $X \rt Y'$ which is not a section factors through $\al$. Dually, one can define a right almost split map.

Let $\CT$ be a triangulated $R$-category, where $R$ is a commutative artinian ring. Assume that $\CT$ is a Krull-Schmidt category and for every $X, Y \in \CT$, $\Hom_{\CT}(X,Y)$ is a finitely generated $R$-module. A triangle $X \st{\al} \rt Y \st{\beta}\rt Z \rightsquigarrow$ in a triangulated category $\CT$ is called an Auslander-Reiten triangle, if $\al$ is left almost split and $\beta$ is right almost split. We say that a triangulated category $\mathcal{T}$ has Auslander-Reiten triangles if for each indecomposable object $Z$ in $\mathcal{T}$ there is an Auslander-Reiten triangle as above.

In view of the following theorem, to show that $\D^{\bb}_{\Gp}(\mmod \La)$ has Auslander-Reiten triangles it is enough to show that $\D^{\bb}(\Fmod (\Gprj \La))$ has Auslander-Reiten triangles. We would like to thank Osamu Iyama for the idea of the following result.

\begin{theorem}\label{Equi}
Let $\La$ be an artin algebra satisfying property $\divideontimes$. Then there is an equivalence
$$\D^{\bb}_{\Gp}(\mmod \La) \simeq \D^{\bb}(\Fmod (\Gprj \La))$$
of triangulated categories.
\end{theorem}

\begin{proof}
Since $\Gprj \La$ is a contravariantly finite subcategory of $\mmod \La$, an argument similar to the proof of \cite[Theorem 3.6]{GZ} works to show that $\D^{\bb}_{\Gp}(\mmod \La) \simeq \K^{-, \Gp \bb}(\Gprj \La)$.

On the other hand, we know that $\D^{\bb}(\Fmod (\Gprj \La)) \simeq \K^{-,\bb}(\Prj \Fmod (\Gprj \La))$. Hence it is sufficient to prove that $\K^{-, \Gp \bb}(\Gprj \La) \simeq \K^{-,\bb}(\Prj \Fmod (\Gprj \La))$. To this end, we define a functor $$\phi: \Gprj \La \lrt \Prj \Fmod (\Gprj \La)$$
by $\phi(G)= \Hom_{\La}(-,G)$, for any finitely generated Gorenstein projective $\La$-module $G$. Note that every object in $\Prj \Fmod (\Gprj \La)$ is of the form $\Hom_{\La}(-, G)$ for some $G \in \Gprj \La$. This fact in conjunction  with Yoneda Lemma implies that $\phi$ is an equivalence. This equivalence can be extended, in a natural way, to the following equivalence
$$\bar{\phi} : \K(\Gprj \La) \lrt \K(\Prj \Fmod (\Gprj \La)).$$
Now observe that, in view of Yoneda Lemma, a sequence
$$ \cdots \rt G^{i-1} \rt G^i \rt G^{i+1} \rt \cdots $$
in $\Gprj \La$ is $\Gp$-acyclic if and only if the sequence
$$\cdots \rt \Hom_{\La}(-,G^{i-1}) \rt \Hom_{\La}(-, G^i) \rt \Hom_{\La}(-, G^{i+1}) \rt \cdots$$
in $\Prj \Fmod(\Gprj \La))$ is acyclic. Therefore, $\bar{\phi}$ restricts to the desired equivalence
\[\bar{\phi} : \K^{-, \Gp \bb}(\Gprj \La) \lrt \K^{-,\bb}(\Prj \Fmod (\Gprj \La)).\]
So the proof is complete.
\end{proof}

Let $\La$ be an artin algebra over a commutative artinian ring $R$. It can be easily seen that the duality $D : \mmod \La \rt \mmod \La^{\op}$ can be extended to the duality $D : \Fmod (\mmod \La) \lrt \Fmod ((\mmod \La)^{\op})$, defined by $D(F)(X)= D(F(X))$, for all $X  \in \mmod \La$.

Motivated by this fact, the notion of dualizing $R$-varieties is introduced by Auslander and Reiten \cite{AR74}. As this concept will be used to prove the existence of Auslander-Reiten triangles in $\D^{\bb}(\Fmod (\Gprj \La))$, we recall briefly the definition. An skeletally small additive category $\CC$ is called variety if idempotents split. An $R$-variety is a variety $\CC$ equipped with an $R$-module structure on abelian group $\Hom_{\CC}(X,Y)$, for each pair of objects $X$ and $Y$ in $\CC$.

Let $\CC$ be a finite $R$-variety, i.e. $\Hom_{\CC}(X,Y)$ is a finitely generated $R$-module for every $X$ and $Y$ in $\CC$. It is easily checked that the functor $D: ( \CC^{\op}, \mmod R) \lrt (\CC , \mmod R)$ given by $D(F)(X)=D(F(X))$, defines a duality. A dualizing $R$-variety is defined to be a finite $R$-variety $\CC$ with the property that for every finitely presented $\CC$-module $F$, the $\CC^{\op}$-module $DF$ is finitely presented as well.
\vspace{0.2cm}

{\bf Theorem.} \cite[Theorem 2.4]{AR74}\label{AR} A finite $R$-variety $\CC$ is  dualizing  if and only if it satisfies the following conditions.
\begin{itemize}
\item[$(i)$] $\CC$ and $\CC^{\op}$ have pseudokernels.
\item[$(ii)$] For any $Y \in \CC$, there is an object $C$ in $\CC$ such that the morphism
$$\Hom_{\CC}(X,Y) \lrt \Hom_{\End_{\CC}(C)^{\op}}(\Hom_{\CC}(C,X), \Hom_{\CC}(C,Y))$$
is an isomorphism for all $X \in \CC$.
\item[$(iii)$] For any $Y \in \CC^{\op}$, there is an object $C$ in $\CC^{\op}$ such that the morphism
$$\Hom_{\CC^{\op}}(X,Y) \lrt \Hom_{\End_{\CC^{\op}}(C)^{\op}}(\Hom_{\CC^{\op}}(C,Y) , \Hom_{\CC^{\op}}(C,X))$$
is an isomorphism for all $X$ in $\CC^{\op}$.
\end{itemize}

\begin{remark}\label{GPrj-dualizing}
Let $\La$ be an artin $R$-algebra satisfying property $\divideontimes$. Then $\Gprj \La$ is a dualizing $R$-variety. To see this note that since $\Gprj \La$ is a contravariantly finite subcategory of $\mmod\La$, by \cite[Corollary 2.6]{KS}, it is a functorially finite subcategory of $\mmod\La$. Now, one may use the fact that the functor $\Hom_{\La}(-,\La)$ induces a duality $\Gprj \La \lrt \Gprj \La^{\op}$, to deduce that $\Gprj \La$ satisfies all the statements of the above theorem.

So, the duality functor $D: \mmod \La \lrt \mmod \La^{\op}$ can be extended to  a duality $D : \Fmod (\Gprj \La)\lrt \Fmod ((\Gprj \La)^{\op})$.
Therefore any injective object in $\Fmod (\Gprj \La)$ is of the form $D \Hom_{\La}(G,-)$, for some $G \in \Gprj \La$.

We define a functor $$\mathcal{V} : \Prj \Fmod (\Gprj \La) \lrt \Inj \Fmod (\Gprj \La)$$ by $\mathcal{V} (\Hom_{\La}(-,G)) = D \Hom_{\La}(G,-)$. Clearly, this functor is an equivalence of categories and extends naturally to an equivalence $$\K(\Prj \Fmod (\Gprj \La)) \lrt \K(\Inj \Fmod (\Gprj \La))$$ of triangulated categories, again denoted by $\mathcal{V}.$
\end{remark}

Next proposition provides a version of Serre duality in the category $\D^{\bb}(\Fmod (\Gprj \La))$.

\begin{proposition}\label{Serre}
Let $\La$ be a Gorenstein algebra. Then, for every $\X, \Y \in \D^{\bb}(\Fmod (\Gprj \La))$,  there is the following isomorphism in $\D^{\bb}(\Fmod (\Gprj \La))$
$$ \Hom(\X, \Y ) \cong D \Hom(\Y , \mathcal{V} \X).$$
\end{proposition}

\begin{proof}
First note that it is known that if $\La$ is Gorenstein, then $\Gprj \La$ is contravariantly finite subcategory of $\mmod \La$. Now, by Lemma \ref{Gor.ness}, $\Fmod (\Gprj \La)$ has finite global dimension. Thus $\D^{\bb}(\Fmod (\Gprj \La)) \simeq \K^{\bb}(\Prj \Fmod (\Gprj \La))$ and so it is enough to prove the desired  isomorphism in $\K^{\bb}(\Prj \Fmod (\Gprj \La))$. Note that complexes $\X$ and $\Y$ in $\K^{\bb}(\Prj \Fmod (\Gprj \La))$ can be considered as $\Hom_{\La}(-, X^{\bullet})$ and $\Hom_{\La}(-, Y^\bullet)$ for some $X^\bullet, Y^\bullet \in \K^{\bb}(\Gprj \La)$, respectively.

Now, the result follows from the following isomorphisms that exist thanks to the Yoneda Lemma
\[\begin{array}{lllll}
D\Hom(\Hom_{\La}(-, Y^\bullet), \mathcal{V} \Hom_{\La}(-,X^\bullet)) & \cong D\Hom(\Hom_{\La}(-,Y^\bullet),D\Hom_{\La}(X^\bullet,-))
\\ & \cong D \Hom(\Hom_{\La}(X^\bullet, -), D \Hom_{\La}(-,Y^\bullet))
\\ & \cong D(D \Hom(X^\bullet,Y^\bullet))
\\& \cong \Hom(X^\bullet,Y^\bullet)
\\ & \cong \Hom(\Hom_{\La}(-,X^\bullet), \Hom_{\La}(-,Y^\bullet)).
\end{array}\]
\end{proof}

\begin{theorem}
Let $\La$ be a Gorenstein algebra. Then the derived category $\D^{\bb}(\Fmod(\Gprj \La))$ admits Auslander-Reiten triangles.
\end{theorem}

\begin{proof}
The same argument as in the proof of \cite[Theorem I. 4.6]{Hap}, or \cite[Theorem I.2.4]{RV}, together with the Proposition \ref{Serre} works to prove the result.
\end{proof}

Following result should be compared with \cite[Corollay 1.5]{Hap2}.

\begin{theorem}\label{AR-triangels}
Let $\La$ be an artin algebra satisfying property $\divideontimes$. Then the Gorenstein derived category $\D^{\bb}_{\Gp}(\mmod \La)$ has Auslander-Reiten triangles if and only if $\La$ is Gorenstein.
\end{theorem}

\begin{proof}
Assume that $\La$ is a Gorenstein algebra. The above theorem and Theorem \ref{Equi} imply that $\D^{\bb}_{\Gp}(\mmod \La)$ admits Auslander-Reiten triangles.

For the converse, let the Gorenstein derived category $\D^{\bb}_{\Gp}(\mmod \La)$ has Auslander-Reiten triangles. Then, by Theorem \ref{Equi}, $\D^{\bb}(\Fmod (\Gprj \La))$ has Auslander-Reiten triangles. Now, the same argument as in the proof of \cite[Theorem 1.4]{Hap2} shows that $\Fmod(\Gprj \La)$ has finite global dimension. Hence, by Lemma \ref{Gor.ness}, $\La$ is Gorenstein.
\end{proof}

Let $\La$ be a noetherian ring. The canonical functor $\K^{\bb}(\prj \La)\lrt \D^{\bb}(\mmod \La)$ is fully faithful. So we can consider $\K^{\bb}(\prj \La)$ as a (thick) triangulated subcategory of $\D^{\bb}(\mmod \La)$. The singularity category of $\La$, $\D_{\rm sg}(\La)$, is defined to be the Verdier quotient of $\D^{\bb}(\mmod \La)$ by $\K^{\bb}(\prj \La)$, see \cite{Buc} and Orlov \cite{Or}.

Moreover, by \cite[Lemma 3.5]{GZ}, the canonical functor $\K^{\bb}(\Gprj \La) \lrt \D^{\bb}_{\Gp}(\mmod \La)$ is fully faithful. So we can consider the Verdier quotient $\D^{\bb}_{\Gp}(\mmod \La) / \K^{\bb}(\Gprj \La)$, that is called the Gorenstein singularity category of $\La$ and is denoted by $\D_{\rm Gsg}(\La)$. Gao and Zhang \cite{GZ} used this category to provide a characterization for Gorenstein rings.

Here we apply Theorem \ref{Equi} and this Gorenstein version of the singularity category to generalise Theorem 4.5 of \cite{KZ} for artin algebras satisfying $\divideontimes$.

Recall that an abelian category $\CA$ is called CM-free, if $\Prj \CA= \GPrj \CA$.
An artin algebra $A$ is called CM-free, if $\mmod A$ is a CM-free category, see \cite{Chen}.

\begin{proposition}\label{CM-free}
Let $\La$ be an artin algebra satisfying $\divideontimes$. Then the category $\Fmod (\Gprj \La)$ is CM-free.
\end{proposition}

\begin{proof}
In view of Theorem \ref{Equi}, we have an equivalence
$$\D^{\bb}_{\Gp}(\mmod \La) \st{\bar{\phi}}\lrt  \D^{\bb}(\Fmod (\Gprj \La))$$
of triangulated categories. So there is a commutative diagram
\[\xymatrix{\D_{\rm Gsg}(\La) \ar[r]^{\st{\widetilde{\phi}}\sim \ \ \ \ \ \ } \ar@{<-}[d] & \D_{\sg}(\Fmod(\Gprj \La)) \ar@{<-}[d] \\ \D^b_{\Gp}(\mmod \La) \ar@{<-_{)}}[d] \ar[r]^{\sim \ \ \ \ \ }  & \D^{\bb}(\Fmod(\Gprj \La))\ar@{<-_{)}}[d] \\ \K^{\bb}(\Gprj \La) \ar[r]^{\sim \ \ \ \ \ \ } & \K^{\bb}(\Prj \Fmod(\Gprj \La)).}\]
Now, \cite[Theorem 3.1]{BJO} implies the existence of the following commutative diagram
\[\xymatrix{\D_{\rm Gsg}(\La) \ar[r]^{\st{\widetilde{\phi}}\sim} \ar@{<-_{)}}[d] & \D_{\sg}(\Fmod(\Gprj\La)) \ar@{<-_{)}}[d] \\ \K_{\rm tac}(\Gprj \La) \ar[r]^{\sim \ \ \ \ \ \ \ }  & \K_{\rm tac}(\Prj \Fmod (\Gprj \La)).}\]
Note that syzygies of complexes in $\K_{\rm tac}(\Gprj \La)$ are Gorenstein projective, see \cite{Hua,SSW}. So complexes in  $\K_{\rm tac}(\Gprj \La) $ are split exact and hence $\K_{\rm tac}(\Gprj \La)$ vanishes. This in turn implies that $\K_{\rm tac}(\Prj \Fmod (\Gprj \La))$ vanishes. But, this is equivalent to say that the category $\Fmod (\Gprj \La)$ is CM-free.
\end{proof}

Let $\La$ be an artin algebra of finite CM-type. So there is a finitely generated Gorenstein projective $\La$-module $G$ such that $\Gprj\La = \add G$. The endomorphism $\End_{\La}(G)$ is then called the Gorenstein Auslander algebra of $\La$. As a direct consequence of the above proposition, we have the following corollary. This result recently has been proved by Kong and Zhang \cite[Theorem 4.5]{KZ}.

\begin{corollary}
Let $\La$ be an artin algebra of finite CM-type. Then the Gorenstein Auslander algebra of $\La$ is CM-free.
\end{corollary}

\begin{proof}
Let $\Gprj \La = \add G$, where $G \in \Gprj \La$. Then the covariant functor $\Hom_{\La}(G,-)$ induces an equivalence $\mmod (\Gprj \La) \simeq \mmod \End_{\La}(G)$. By Proposition \ref{CM-free}, $\mmod (\Gprj \La)$, and hence $\End(G)$, is CM-free.
\end{proof}

\section{Equivalence of $\D^{\bb}_{\Gp}(\mmod\La)$ and $\D^{\bb}_{\Gi}(\mmod\La)$}\label{6}
Let $\CA$ be an abelian category. One can define the concept of $\CY$-coacyclic complexes for a subcategory $\CY$ of $\CA$, i.e. a complex $\A$ such that the induced complex $\Hom_{\CA}(\A, Y)$ is acyclic for every $Y\in \CY$. Also, a $\CY$-quasi-isomorphism is  a chain map $f : \A \lrt {\bf B}$ such that $\Hom_{\CA}(f, Y)$ is a quasi-isomorphism for all $Y \in \CY$. So, we have the relative derived category $\D_{\CY}(\CA)$, defined as $\D_{\CY}(\CA) := \K(\CA) /\K_{\CY \mbox{-} {\rm ac}}(\CA)$, where $\K_{\CY \mbox{-} {\rm ac}}(\CA)$ denotes the homotopy category of $\CY$-coacyclic complexes.

If $\CA$ has enough injective objects and $\CY=\GInj \CA$, then the relative derived category $\D_{\GI}(\CA)$ is called Gorenstein (injective) derived category of $\CA$. In case $ \CA= \mmod \La$ and $\CY=\Ginj \La$, where $\La$ is a ring, we write $\K_{\Gi \mbox{-}\ac}(\mmod\La)$ for $\K_{\CY\mbox{-}\ac}(\CA)$ and $\D_{\Gi}(\mmod\La)$ for $\D_{\CY}(\CA)$.

Let $\La$ be a Gorenstein ring. By  \cite[Corollary 12.3.5]{EJ} every short exact sequence is $\GP$-acyclic if and only if it is $\GI$-coacyclic. This, in turn, implies that in this case $\D_{\GP}(\Mod \La)\simeq \D_{\GI}(\Mod \La)$.
Our aim in this section is to get conditions for the equivalences of the triangulated categories $\D^{\bb}_{\Gp}(\mmod \La)$ and $\D^{\bb}_{\Gi}(\mmod \La)$.

Given an artin algebra $\La$, recall that  the Nakayama functor on $\mmod \La$ is defined as $$\nu := D\Hom_{\La}(-,\La): \mmod \La \lrt \mmod \La.$$ Restriction of the Nakayama functor to the category $\prj \La$ induces an equivalence $\nu : \prj \La \lrt \inj \La$. The quasi-inverse of this equivalence is given by $$\nu^{-}:= \Hom_{\La}(D\La,-): \inj \La \lrt \prj \La.$$ Note that, if one consider $\nu^-: \mmod \La \lrt \mmod \La$, then $(\nu, \nu^-)$ is an adjoint pair. Recently, it has been shown by Beligiannis \cite[Proposition 3.4]{Be2} that the Nakayama functor extends  to equivalence $\nu : \Gprj \La \lrt \Ginj \La$ of categories with the quasi-inverse $\nu^-$.

\begin{theorem}\label{Thm6.1}
Let $\La$ be an artin algebra of finite CM-type. Then there is an equivalence
\[\D_{\Gp}^{\bb}(\mmod \La) \simeq \D_{\Gi}^{\bb}(\mmod \La)\]
of triangulated categories.
\end{theorem}

\begin{proof}
Since  $\La$ is an artin algebra of finite CM-type, there is a finitely generated Gorenstein projective  $\La$-module $G$ such that $\Gprj \La = \add G$ and $\Ginj \La= \add \nu G$. Clearly, $\add G$, resp. $\add \nu G$, is a contravariantly, resp. covariantly, finite subcategory of $\mmod \La$.

Hence by Theorem \ref{Equi} there exist the following equivalences of triangulated categories
\[\begin{array}{ll}
\D^{\bb}_{\Gp}(\mmod \La) & \simeq \D^{\bb}(\Fmod (\Gprj \La)) \\
& \simeq \D^{\bb}(\Fmod (\add G))
\end{array}
\]

Now, the covariant functor $\Hom_{\La}(G,-)$ induces an equivalence
$$\D^{\bb}(\Fmod (\add G)) \simeq \D^{\bb}(\mmod \End(G))$$
of triangulated categories.

The same argument as in the proof of Theorem 3.6 of \cite{GZ} works to obtain a triangle equivalence
$$\D^{\bb}_{\Gi}(\mmod \La) \simeq \K^{+, \Gi \bb}(\Ginj \La),$$
where $\K^{+, \Gi \bb}(\Ginj \La)$ is the full subcategory of $\K^{+}(\Ginj \La)$ consisting of all complexes $\X$ such that there is an integer $n$ in which ${\rm H}^i(\Hom_{\La}(\X, G))=0$, for all $i \geq n$ and every $G \in \Ginj \La$. The Yoneda functor induces a duality
$$\K^{+, \Gi \bb}(\Ginj \La) \simeq \K^{-, \bb}(\Prj \Fmod (\Ginj \La)^{\op}),$$
where $\Fmod ((\Ginj \La)^{\op})$ denotes the category of covariant functors from $\Ginj \La$ to $\CA b$.

On the other hand, there is an equivalence
$$\D^{\bb}(\mmod (\Ginj \La)^{\op}) \simeq \K^{-, \bb}(\Prj \Fmod (\Ginj \La)^{\op})$$
of triangulated categories. All together, we have an equivalence
$$\D^{\bb}_{\Gi}(\mmod \La) \simeq \D^{\bb}(\Fmod( (\Ginj \La)^{\op}))^{\op}$$
of triangulated categories.
 Moreover, the covariant functor $\Hom_{\La}(\nu G, -)$ gives a triangle equivalence
$$\D^{\bb}(\Fmod ((\Ginj \La)^{\op})) \simeq \D^{\bb}(\mmod \End_{\La}(\nu G)^{\op}).$$
 Thus,  we have an equivalence
$$\D_{\Gi}^{\bb}(\mmod \La) \simeq \D^{\bb}(\mmod \End_{\La}(\nu G)^{\op})^{\op}$$
of triangulated categories.

For an artin algebra $\Ga$, the duality $D: \mmod \Ga \lrt \mmod \Ga^{\op}$ can be extended to the duality $\R D : \D^{\bb}(\mmod \Ga) \lrt \D^{\bb}(\mmod \Ga^{\op})$ of bounded derived categories. Therefore, there is an equivalence
\[\D^{\bb}(\mmod \End_{\La}(G)^{\op})^{\op} \lrt \D^{\bb}(\mmod \End_{\La}(G)).\]
Now, the fact that $\End_{\La}(\nu G) \cong \End_{\La}(G)$ completes the proof.
\end{proof}

In the following we intend to prove this result for a virtually Gorenstein algebra $\La$. By the preceding discussion, we need to construct a duality between $\Fmod (\Gprj \La)$ and $\Fmod ((\Ginj \La)^{\op})$. To this end, we need the following lemma.

\begin{lemma}
Let $\La$ be a virtually Gorenstein algebra. Then the duality $D: \mmod \La \lrt \mmod \La^{\op}$ induces the duality
$$D: \Fmod (\Ginj \La) \lrt \Fmod ((\Ginj \La)^{\op}).$$
\end{lemma}
\begin{proof}
By \ref{GPrj-dualizing}, $\Gprj \La$ is a dualizing $R$-variety. On the other hand, there is an equivalence $\nu : \Gprj \La \lrt \Ginj \La$. Therefore, $\Ginj \La$ has the properties just derived for $\Gprj \La$ and so $\Ginj \La$ is a dualizing $R$-variety. Hence we have the desired duality.
\end{proof}

Observe that the Nakayama functor $\nu: \Ginj \La \st{\sim}\lrt \Gprj \La$ induces the equivalence $$\nu: \Fmod(\Ginj \La) \st{\sim}\lrt  \Fmod (\Gprj \La).$$ So in view of the above lemma, there  is the following duality of triangulated categories
$$\eta: \Fmod (\Gprj \La) \st{\nu}\lrt \Fmod (\Ginj \La) \st{D} \lrt \Fmod((\Ginj \La)^{\op}).$$

Let $\La$ be a virtually Gorenstein algebra. By Theorem \ref{Equi}, $$\D^{\bb}_{\Gp}(\mmod \La) \simeq \D^{\bb}(\Fmod (\Gprj \La)).$$ The same argument as in the proof of Theorem \ref{Thm6.1} can be applied to show that there is an equivalence
$$\D^{\bb}_{\Gi}(\mmod \La) \simeq \D^{\bb}(\Fmod ((\Ginj \La)^{\op}))^{\op}$$
of triangulated categories. We shall use these facts in the proof of the following theorem.

\begin{theorem}\label{ThmSec6}
Let $\La$ be a virtually Gorenstein  algebra. Then there is an equivalence
$$\D_{\Gp}^{\bb}(\mmod \La) \simeq \D_{\Gi}^{\bb}(\mmod \La),$$
of triangulated categories.
\end{theorem}

\begin{proof}
As we mentioned before, there are equivalences
$$\D^{\bb}_{\Gp}(\mmod \La) \simeq \D^{\bb}(\Fmod (\Gprj \La)) \ \  \text{and} \ \ \D^{\bb}_{\Gi}(\mmod \La) \simeq \D^{\bb}(\Fmod ((\Ginj \La)^{\op}))^{\op}$$
of triangulated categories.

Also, the duality  $ \eta:  \Fmod (\Gprj \La)\lrt \Fmod((\Ginj \La)^{\op}) $  can be extended to the duality $\R \eta: \D^{\bb}(\Fmod (\Gprj \La)) \lrt \D^{\bb}(\Fmod((\Ginj \La)^{\op}))$. The proof is now complete.
\end{proof}

\section{Reflection functors}
Let $\CQ$ be a quiver and $i$ be a vertex  of $\CQ$. We let $\sigma_i\CQ$ denote the quiver which is obtained from $\CQ$ by reversing all arrows which start or end at $i$.

It is proved in \cite[I. 5.7]{Hap} that if $\La$ is a field, then $\D^b(\CQ) \simeq \D^b(\sigma_i\CQ)$, where $\D^{\bb}(\CQ)$, resp. $\D^{\bb}(\sigma_i\CQ)$, denotes the bounded derived category of finitely generated representations of $\CQ$, resp. $\sigma_i\CQ$,  over $\La$.

This section is devoted to prove the relative version of this result in terms of the Gorenstein derived categories.
\vspace{0.2cm}

\s A quiver $\CQ$ is a directed graph which is denoted by a quadruple $\CQ=(\CQ_0,\CQ_1,s,t)$, where $\CQ_0$, resp. $\CQ_1$, is the set of vertices, resp. arrows, of $\CQ$ and $s,t: \CQ_1 \rt \CQ_0$ are maps which associate to any arrow $\alpha \in \CQ_1$
its source $s(\al)$ and its target $t(\al)$. A quiver $\CQ$ is called finite if both $\CQ_0$ and $\CQ_1$ are finite sets.
We may consider $\CQ$ as a category: the objects are vertices and morphisms are arrows. A representation of $\CQ$ over a ring $\La$ is a covariant functor $\CM: \CQ \rt \Mod \La$.
We denote the category of all, resp. all finitely generated, representations  of $\CQ$ in $\Mod \La$ by ${\rm Rep}(\CQ,\La)$, resp. ${\rm rep}(\CQ,\La)$. It is known that ${\rm Rep}(\CQ,\La)$ is a Grothendieck category which has enough projective objects. Let $\CM \in {\rm Rep}(\CQ, \La)$ be a representation of $\CQ$. For every vertex $v \in \CQ_0$, $\CM_v$ denotes the module at vertex $v$.
\vspace{0.2cm}

As in the absolute case, we write $\D_{\Gp}(\CQ)$, resp. $\D_{\sg}(\CQ)$,  instead of  $\D_{\Gp}({\rm rep}(\CQ,\La))$, resp. $\D_{\sg}({\rm rep}(\CQ,\La))$. Also, we denote by $\inj \CQ$, resp. $\prj \CQ$, $\Gprj \CQ$, the full subcategory of ${\rm rep}(\CQ, \La)$ formed by all  injective, resp.  projective, Gorenstein projective, representations of $\CQ$.

Let us recall the definition of a pair of reflection functors. For more details consult \cite[3.3]{K08}.
Let $i$ be a sink of $\CQ$. The reflection functor $S_i^+: {\rm Rep}(\CQ, \La) \lrt {\rm Rep}(\CQ,\La)$ is defined as follows: let $\CM$ be an arbitrary representation of $\CQ$. Then $S_i^+(\CM)_v=\CM_v$ for any vertex $v \neq i$, and $S_i^+(\CM)_i$ is the kernel of a map $\bigoplus_{t(a)=i}\CM_{s(a)} \st{\zeta}\rt \CM_i$.

Similarly, if $i$ is a source of $\CQ$, we have the reflection functor $S_i^-: {\rm Rep}(\CQ, \La) \lrt {\rm Rep}(\CQ,\La)$ which is defined by $S_i^-(\CM)_v=\CM_v$ for any $v \neq i$, and $S_i^-(\CM)_i$ is the cokernel of a map
$ \CM_i \rt \bigoplus_{s(a)=i} \CM_{t(a)}.$

\begin{lemma}\label{AHV2}
Let $\La$ be an artin algebra and $\CQ$ be a finite acyclic quiver. Let $i$ be either a sink or a source of $\CQ$. Then there is an equivalence
$$\D^{\bb}(\CQ) \simeq \D^{\bb}(\sigma_i\CQ),$$
of triangulated categories.
\end{lemma}
\begin{proof}
Following similar argument as in \cite[Theorem 3.2.7]{AHV2}, we have the following two equivalences
$$\K^{\bb}(\inj \CQ) \st{S_i^+} \lrt \K^{\bb}(\inj \sigma_i\CQ)$$
$$\K^{\bb}(\prj \CQ) \st{S_i^-}\lrt \K^{\bb}(\prj \sigma_i\CQ),$$
where in the first equivalence $i$ is a sink and in the second one $i$ is a source of $\CQ$.

On the other hand, for any artin algebra, the Nakayama functor induces an equivalence between the category of finitely generated projectives and finitely generated injectives. Since $\La \CQ$ and $\La \sigma _i \CQ$ are artin algebras, there exist equivalences $\inj \CQ \simeq\prj \CQ$ and $\inj \sigma_i\CQ \simeq \prj \sigma_i\CQ$. These equivalences can be extended, naturally, to the following equivalences
 $$\K^{\bb}(\inj \CQ) \simeq \K^{\bb}(\prj \CQ) \ \ \ \text{ and } \ \ \ \K^{\bb}(\inj \sigma_i \CQ) \simeq \K^{\bb}(\prj \sigma_i \CQ).$$
 So, in both cases, we have an equivalence
$$\K^{\bb}(\prj \CQ) \lrt \K^{\bb}(\prj \sigma_i\CQ)$$
of triangulated categories. By \cite[Theorem 6.4 and Proposition 8.2]{Ric}, it is equivalent to the existence of following triangulated equivalence
$$\D^{\bb}(\CQ) \st{\sim}\lrt \D^{\bb}(\sigma_i\CQ).$$
\end{proof}

As a direct consequence of the above lemma we have the following corollaries.
\begin{corollary}\label{Cor1}
Let $\La$ be an artin algebra and $\CQ$ be a finite acyclic quiver. Then there is an equivalence
$$\D_{\sg}(\CQ) \cong \D_{\sg}(\sigma_i\CQ),$$
of triangulated categories.
\end{corollary}

\begin{proof}
Observe that if two noetherian rings are derived equivalent, then they have the same singularity category. So Lemma \ref{AHV2} implies the result.
\end{proof}

\begin{corollary}\label{GP}
Let $\La$ be a Gorenstein algebra and $\CQ$ be a finite acyclic quiver. Then there is an equivalence
$$\underline{\Gp}\mbox{-}\CQ \simeq \underline{\Gp}\mbox{-}\sigma_i\CQ,$$
of triangulated categories.
\end{corollary}

\begin{proof}
By \cite{Buc, Hap91}, $\D_{\sg}(R) \simeq \underline{\Gp}\mbox{-}R$, provided $R$ is a Gorenstein ring. Since $\La$ is Gorenstein and $\CQ$, resp. $\sigma_i\CQ$, is finite and acyclic, $\La\CQ$, resp. $\La\sigma_i\CQ$, is Gorenstein. Now, the result follows from Corollary \ref{Cor1}.
\end{proof}

\begin{corollary}\label{Gor-Ref}
Let $\La$ be a Gorenstein algebra and $\CQ$ be a finite acyclic quiver. Then $\K^{\bb}(\Gprj \CQ)\simeq \K^{\bb}(\Gprj \sigma_i\CQ)$. In particular, we have an equivalence $\D^{\bb}_{\Gp}(\CQ) \simeq \D^{\bb}_{\Gp}(\sigma_i\CQ)$.
\end{corollary}

\begin{proof}
By Lemma \ref{AHV2}, $\D^{\bb}(\CQ) \simeq \D^{\bb}(\sigma_i\CQ)$. If we let $\La$ to be Gorenstein, then so are $\La \CQ$ and $\La\sigma_i\CQ$. So  \cite[Theorem 5.2.3]{AHV} implies that  $\K^{\bb}(\Gprj \CQ) \simeq \K^{\bb}(\Gprj \sigma_i\CQ)$. Moreover,  Corollary 3.8 of \cite{GZ} follows the equivalence
$$\D^{\bb}_{\Gp}(\CQ) \simeq \D^{\bb}_{\Gp}(\sigma_i\CQ),$$
of triangulated categories.
\end{proof}

\begin{corollary}
Let $\La$ be a Gorenstein algebra and $T$ be a finite tree. If   $\CQ_1$ and $\CQ_2$ are two  acyclic quivers with the same underlying graph as $T$, then
\begin{itemize}
\item[$(i)$] $\underline{\Gp}\mbox{-}\CQ_1 \simeq \underline{\Gp}\mbox{-}\CQ_2$.

\item[$(ii)$] $\D_{\Gp}^{\bb}(\CQ_1) \simeq \D_{\Gp}^{\bb}(\CQ_2)$.

\end{itemize}
\end{corollary}
\begin{proof}
By \cite[I. 5.7]{Hap}, $\CQ_2$ can be obtained from $\CQ_1$ by a finite sequence of reflection functors. Now, the equivalences follow from Corollaries  \ref{GP} and \ref{Gor-Ref}.
\end{proof}

\section*{Acknowledgments}
The authors would like to thank the referees for their useful hints and comments that improved our exposition. The first author also thanks the IMU-Simons Foundation Travel Fellowship for providing a travel grant to visit MPIM. We also thank the Center of Excellence for Mathematics (University of Isfahan).

\end{document}